\documentclass{amsart}
\usepackage[utf8]{inputenc}

\usepackage{xcolor}
\usepackage{amsmath}
\usepackage{amssymb}
\usepackage{comment}

\usepackage{geometry}
\allowdisplaybreaks

\title{$L^p$ approximation results for infinite dimensional Neural Networks}
\author{Luca Galimberti}
\address{Luca Galimberti \\ 
King's College London\\
Department of Mathematics\\
Strand Building WC2R 2LS, London, UK}
\email[]{luca.galimberti\@@kcl.ac.uk}

\newtheorem{theorem}{Theorem}[section]
\newtheorem{definition}[theorem]{Definition}
\newtheorem{lemma}[theorem]{Lemma}
\newtheorem{proposition}[theorem]{Proposition}

\newtheorem{remark}[theorem]{Remark}
\newtheorem{example}[theorem]{Example}

\newcommand{\R}{\mathbb R}
\newcommand{\C}{\mathbb C}
\newcommand{\N}{\mathbb N}

\newcommand{\norm}[1]{\left\lVert#1\right\rVert}
\newcommand{\abs}[1]{\left |#1\right|}
\newcommand{\X}{\mathfrak X}

\DeclareMathOperator{\Span}{span}

\keywords{Neural networks, global universal approximation, Fr\'echet space, quasi-Polish space}

\begin{document}

\begin{abstract}
Leveraging the neural architectures which we introduced in \cite{benth2023neural}, we show a global universal approximation theorem in the topology of $L^p(\mu)$, where $1\le p<\infty$ and $\mu$ is a Radon probability measure on a suitable infinite dimensional topological space $\X$. Namely, any function $f:\X\to \R$ in $L^p(\mu)$ can be approximated to any degree of accuracy by suitable infinite dimensional architectures. These architectures can be in turn approximated by almost classical neural networks which are specified by a finite number of parameters only. The vectorial case (where $f=f(x)\in E$ and $E$ is a Banach space) is also considered and analogous results are obtained.
\end{abstract}
\maketitle

\section{Introduction}

The classical universal approximation theorem asserts that every continuous function from $\mathbb{R}^n$ to $\mathbb{R}$ can be approximated arbitrarily well by a neural network with a single hidden layer. To state this more precisely, let $\sigma:\mathbb{R}\to\mathbb{R}$ be a fixed continuous function. Given parameters $a\in\mathbb{R}^n$ and $\ell,b\in\mathbb{R}$, define a {\em neuron} as the function $\mathcal{N}_{\ell,a,b}\in C(\mathbb{R}^n;\mathbb{R})$ given by
\[
x\longmapsto \ell\sigma(a^\top x+b).
\]
The universal approximation theorem identifies conditions on the {\em activation function} $\sigma$ under which the linear space generated by these neurons,
\[
\mathfrak N(\sigma):=\Span\{\mathcal{N}_{\ell,a,b};\ \ell,b\in\mathbb{R},\ a\in\mathbb{R}^n\},
\]
is dense in $C(\mathbb{R}^n;\mathbb{R})$ with respect to the topology of uniform convergence on compact sets. Equivalently, for every function $f\in C(\mathbb{R}^n;\mathbb{R})$, every compact set $K\subset\mathbb{R}^n$, and every $\varepsilon>0$, there exist $N\in\mathbb{N}$ together with parameters $\ell_i,b_i\in\mathbb{R}$ and $a_i\in\mathbb{R}^n$, $i=1,\ldots,N$, such that
\[
\sup_{x\in K}\left|f(x)-\sum_{i=1}^{N}\mathcal{N}_{\ell_i,a_i,b_i}(x)\right|\leq\varepsilon.
\]

One of the best-known sufficient conditions guaranteeing the density of $\mathfrak N(\sigma)$ was established independently by Cybenko \cite{Cybenko1989} and by Hornik, Stinchcombe, and White \cite{HORNIK1989359}. They showed that density holds whenever $\sigma$ is {\em sigmoid}, that is,
\[
\lim_{t\to\infty}\sigma(t)=1
\qquad\text{and}\qquad
\lim_{t\to-\infty}\sigma(t)=0.
\]
Subsequent work relaxed this assumption. Funahashi \cite{FUNAHASHI1989183} replaced the sigmoid requirement by a boundedness condition, while Leshno {\it et al.} \cite{LESHNO1993861} proved that it suffices for $\sigma$ to be non-polynomial. 

In \cite{hornik1991approximation} Hornik proved among other things a universal approximation theorem which is global (i.e. on the whole space) and holds with respect to the $L^p(\mu)$ topology, where $\mu$ is a finite measure on $\R^n$ and $1\le p<\infty$. In the present work, moving from this last result and the neural architectures which we have introduced in \cite{benth2023neural} and \cite{galimberti2024neural}, we prove several universal approximation theorems at once which are global and hold in the $L^p(\mu)$ sense for any Radon probability measure $\mu$ on $\X$. Here $\X$ (the underlying space) can be a Fr\'echet space or a $\sigma$-compact quasi-Polish space (we will make this notion more precise later). Besides, we do not restrict ourselves to the scalar case only, but we will treat also the problem of approximating via infinite dimensional neural architectures vector-valued functions $f=f(x)\in E$, where $(E,\norm{\cdot}_E)$ is a Banach space. Finally, we address the issue of constructing suitable architectures able to process infinite dimensional input/output pairs but which are specified only by a finite number of parameters.

{\bf Related literature:} The approximation of nonlinear mappings between infinite dimensional spaces has a long history in approximation theory, functional analysis, and machine learning. Classical universal approximation theorems established that finite dimensional feedforward neural networks can approximate continuous functions on compact subsets of $\mathbb{R}^n$ \cite{Cybenko1989,hornik1991approximation,FUNAHASHI1989183}, motivating the extension of these results to infinite dimensional settings. The first instance thereof probably goes back to Sandberg \cite{76498} where in the context of discrete time systems, non-linear functionals on a space of functions from $\mathbb{N} \cup \{ 0\}$ to $\mathbb{N}$ are approximated with neural networks. In Chen and Chen \cite{286886, 392253} the authors consider the approximation of non-linear operators defined on infinite dimensional spaces and use these results for approximating the output of dynamical systems. Among other results they approximate functions $f: K\subset \X \to \R$, where $\X$ is Banach, $K$ is compact and $f$ is continuous. In Mhaskar and Hahm \cite{MhaskarHahm} the authors derive networks that approximate the functionals on the function spaces $L^p ([-1,1]^s)$ for $1 \leq p < \infty $ and $C ([-1,1]^s)$ for integer $s\geq 1$. The common trait of all these works is that they showed that neural networks can approximate continuous operators under suitable compactness assumptions, providing the first rigorous foundation for learning infinite-dimensional input--output relationships in Hilbert or Banach spaces.

The emergence of scientific machine learning renewed interest in learning nonlinear operators directly from data consisting of pairs of input and output functions. Rather than approximating a single solution associated with one discretization, operator learning seeks to approximate the solution operator itself, enabling rapid evaluation for previously unseen inputs. The first widely adopted neural architecture explicitly designed for operator learning was the \emph{Deep Operator Network} (DeepONet) introduced by Lu \emph{et al.}~\cite{LiDeepONets}, and Lauthaler, Mishra and Karniadakis \cite{10.1093/imatrm/tnac001}. Inspired by branch--trunk decompositions and the universal approximation theorem for operators, DeepONet combines a branch network encoding sampled input functions with a trunk network representing evaluation locations. The accompanying theoretical analysis established a universal approximation theorem for continuous nonlinear operators acting on compact subsets of Banach spaces, thereby providing one of the first practically implementable neural architectures with rigorous approximation guarantees. Numerous variants have subsequently been proposed, including physics-informed DeepONets \cite{wang2021learning}. Another breakthrough followed with the introduction of the \emph{Fourier Neural Operator} (FNO) by Li \emph{et al.}~\cite{li2020fourier}. Instead of learning pointwise mappings, FNO parameterizes integral kernels through Fourier multipliers, producing architectures that are discretization invariant and capable of transferring across spatial resolutions. The resulting computational efficiency and scalability enabled the solution of challenging high-dimensional PDE families, making FNO one of the most influential operator-learning architectures. 

Parallel to algorithmic developments, a substantial theoretical literature has recently emerged. Universal approximation theorems have been established for broad classes of neural operator architectures, while approximation rates have been derived under assumptions on operator regularity, smoothness, and intrinsic dimension. For instance, in \cite{Hypertransformer} the authors have considered hyper-transformers architectures with the aid of which they have been able to approximate maps $f:K\subset\R^d\to (Y,\rho)$, whereas $K$ is a compact subset, $(Y,\rho)$ is a suitable metric space and $f$ is assumed to be $\alpha$-H\"older, $0<\alpha\le 1$. Always very recently, \cite{CuchieroSchmockerTeichmann} have studied approximation capabilities of neural networks defined on infinite dimensional weighted spaces, obtaining global universal approximation results for continuous functions whose growth is controlled by a weight function. Crucial step in their proof (as well as in previous classical results) is some version of Stone-Weierstrass theorem.

Infinitely wide neural networks, with an infinite but countable number of nodes in the hidden layer have been studied in the context of Bayesian learning, Gaussian processes and kernel methods by several authors, see e.g., Neal \cite{Neal}, Williams \cite{Williams}, Cho and Saul \cite{ChoSaul} and Hazan and Jaakola \cite{HJ}. Mashkar and Hahm \cite{MhaskarHahm} provides approximation results for such infinitely wide networks. Guss and Salakhutdinov \cite{Guss} prove the universal approximation property for two-layer infinite dimensional neural networks. They show their approximation property for continuous maps between spaces of continuous functions on compacts. Always very recently in \cite{ismailov2026shallow} feedforward neural networks with inputs from a topological space
(TFNNs) are studied, and a universal approximation theorem for shallow TFNNs  demonstrating their capacity to approximate any continuous function defined on this topological space is established. \cite{bilokopytov2026universal}, moving from the architectures we have introduced in \cite{benth2023neural}, proves that neural networks are universal approximators on suitable infinite dimensional spaces whenever the activation function is continuous and non-polynomial, a result which resembles the classical result in $\R^n$.

{\bf Outline} The outline for the paper is as follows. In Section~\ref{main:approx} we derive our first main result, Theorem~\ref{thm: main}, which shows that if $\sigma$ has the so-called {\em separating} property, which has been first introduced in \cite{benth2023neural}, then $\mathfrak N(\sigma)$, the space of neural architectures we are considering (see \eqref{eq: N sigma}), is dense in $L^p(\mu)$ for any $1\le p<\infty$ and any $\mu$ Radon probability measure on $\X$. In Section~\ref{sec: Finite-dimensional approximation} we face the issue of approximating these $L^p$ neural networks which are apriori specified by infinite-number of parameters via suitable architectures which are instead determined by a finite number of parameters only. This culminates in Proposition \ref{prop: finite dimensional approx, Banach} which requires the underlying Banach space $\X$ to be endowed with a Schauder basis in order to obtain neural architectures which are implementable into a machine. In Section \ref{sec: qp}  we largely extend the scope of these neural networks and prove an universal approximation theorem for a wide class of input spaces $\X$, namely the class of (sigma-compact) quasi-Polish spaces with quasi-Polish neural architectures, that have been introduced in \cite{galimberti2024neural}; similarly to Proposition \ref{prop: finite dimensional approx, Banach}, we show also for these neural architectures the possibility of them being implemented in a machine with finite computing power and memory. In Section \ref{sec: general codomain} we move from the problem of approximating functions from $\X$ to $\R$ to the question of approximating in the $L^p$ sense maps $f:\X\to E$ with a more general co-domain, namely we allow $E$ to be a Banach space and we will work in the framework of Bochner spaces $L^p(\mu;E)$. Finally, in Section \ref{sec: diffusion} we discuss an application from diffusion models and denoising diffusion operators, where the necessity of approximating in the $L^p$ sense maps $f$ with infinite dimensional domain and co-domain with suitable neural architectures appear.

\section{An abstract approximation result}\label{main:approx}

Let $\mathfrak X$ be a real Fr\'echet space, namely a locally convex metrizable complete space, and let $(p_k)_{k\in\N}$ be an increasing sequence of seminorms that generates the topology of $\mathfrak X$. We can then consider a metric $d$ on $\mathfrak X$ (that generates the same topology) given by
\begin{equation}\label{metric:loc:conv:space}
    d(x,y):= \sum_{k=1}^\infty 2^{-k}\frac{p_k(x-y)}{1+p_k(x-y)}, 
\end{equation}
for $x,y\in \mathfrak X$. The Borel sigma-algebra of $\X$ will be denoted as $\mathcal{B}(\X)$. We then fix a Radon probability measure $\mu$ on $\X$, namely $\mu$ is a Borel probability measure such that for every $B\in\mathcal B(\X)$ and $\varepsilon>0$, there exists a compact set $K_\varepsilon\subset B$ such that $\mu(B\setminus K_\varepsilon)<\varepsilon$: refer to e.g. \cite{bogachev2007measure}.

Let us consider $\sigma: \X \to \X$ an arbitrary Borel function, which will be termed the \emph{activation function}. Let $A:\X\to \X$ be in $\mathcal{L}(\X)$, i.e. a linear and continuous operator, $b\in\X$ and $\ell\in \X'$, where $\X'$ denotes the topological dual of $\X$. Let us consider the following function:
\begin{equation}\label{eq: definition of 1 layer NN}
    \mathcal{N}_{\ell,A,b} : \X\to \mathbb R, \quad \mathcal{N}_{\ell,A,b}(x):= \langle \ell,\sigma(Ax+b)\rangle, \quad x\in\X,
\end{equation}
where $\langle\cdot,\cdot\rangle$ is the canonical pairing between $\X'$ and $\X$. We will call such function a {\it neuron} with parameters $\ell, A$ and $b$.
Observe that in contrast to \cite{benth2023neural} a neuron $\mathcal{N}_{\ell,A,b}$ will in general not be continuous (but only Borel), as we are not assuming that the activation function is continuous.

As customary, we define the space of the architectures based on the given activation function $\sigma$ as
\begin{equation}\label{eq: N sigma}
    \mathfrak N(\sigma) := \Span\{
    \mathcal{N}_{\ell,A,b}; \,\,\ell\in\X', A\in\mathcal{L}(\X),b\in\X    \},
\end{equation}
namely, we consider all linear combinations of the form
\begin{equation*}
    \sum_{j=1}^N \alpha_j \mathcal{N}_{\ell_j,A_j,b_j}, \quad \alpha_j\in\mathbb R,N\in\N.
\end{equation*}
The maps $\mathcal{N}_{\ell_1,A_1,b_1}, \dots , \mathcal{N}_{\ell_N,A_N,b_N}$ build a {\it hidden layer} with $N$ neurons.

As our aim is to work in an ``$L^p$ setting'' $(1\le p < \infty)$, we have to ensure first that $\mathfrak N(\sigma)$ is a subset of $L^p(\mu)$. We then impose a boundedness condition on the range $\sigma$. First, we recall that an arbitrary set $A\subset\mathfrak{X}$ is von Neumann-bounded if for any $k\in \N$ there exists $c_k>0$ such that $\sup_{x\in A}p_k(x)\leq c_k$. 
We then assume that the range of $\sigma$  
\begin{equation}
    \sigma(\X)\subset \X
\end{equation}
is von Neumann-bounded, which is consistent with the setting of the finite-dimensional version of the $L^p$ universal approximation theorem of \cite{hornik1991approximation}. Since $\ell$ is a continuous functional and in view of the von Neumann-boundedness assumption on $\sigma$, for any $\ell\in\X',A\in\mathcal L(\X),b\in\X$
\begin{equation*}
    \abs{\mathcal{N}_{\ell,A,b}(x)}\leq C_\ell \,p_{k_\ell} (\sigma(Ax+b)),\quad x\in \X
\end{equation*}
for some constant $C_\ell\geq 0$ and $k_\ell\in \N$ (compare Schaefer \cite[Thm. 1.1, p. 74]{schaefer1971topological}), and thus, for a constant $C(\ell,\sigma)$ depending on $\ell$ and $\sigma$ only
\begin{equation}\label{eq: neurons are bounded}
    \abs{\mathcal{N}_{\ell,A,b}(x)}\leq C(\ell,\sigma),\quad x\in \X.
\end{equation}
Namely, all the neurons based on the activation function $\sigma$ are real Borel bounded functions, and therefore for any element
\begin{equation*}
    \sum_{j=1}^N \alpha_j \mathcal{N}_{\ell_j,A_j,b_j}, \quad \alpha_j\in\mathbb R,N\in\N
\end{equation*}
we will have
\begin{equation*}
    \abs{\sum_{j=1}^N \alpha_j \mathcal{N}_{\ell_j,A_j,b_j}(x)} \le C,\quad x\in\X
\end{equation*}
where the constant will clearly depend on $\alpha_j,\ell_j,j=1,\dots,N$ and $\sigma$. Since $\mu$ is a probability measure, it then follows
\begin{equation*}
    \mathfrak N(\sigma) \subset L^p(\mu), \quad 1\le p<\infty,
\end{equation*}
i.e. $\mathfrak N(\sigma)$ is a linear subspace of $L^p(\mu)$.

In order to show that $\mathfrak N(\sigma)$ is dense in $L^p(\mu)$ with respect to the $L^p(\mu)$ topology, we will impose the following technical condition, introduced in \cite{benth2023neural}, namely the so called separating property, which can be seen as the infinite-dimensional counterpart to the well known sigmoidal property for functions from $\mathbb{R}$ to $\mathbb{R}$ (see Cybenko \cite{Cybenko1989}):
\begin{definition}{Separating property:}\label{sigmoid}
There exist $\psi\in \X'\setminus\{0\}$ and $u_+,u_-,u_0\in\X$ such that either $u_+ \notin \Span \{u_0,u_-\}$ or $u_- \notin \Span \{u_0,u_+ \}$ and such that 
\begin{equation}\label{eq: abstract condition on sigma}
\begin{cases}
\lim_{\lambda\to\infty} \sigma(\lambda x) = u_+, \text{ if } x\in \Psi_+\\
\lim_{\lambda\to\infty} \sigma(\lambda x) = u_-, \text{ if } x\in \Psi_-\\
\lim_{\lambda\to\infty} \sigma(\lambda x) = u_0, \text{ if } x\in \Psi_0\\
\end{cases}
\end{equation}
where we have set 
\begin{equation*}
    \Psi_+ =\{ x\in\X; \langle\psi,x\rangle >0 \}, \quad \Psi_- =\{ x\in\X; \langle\psi,x\rangle <0 \}
\end{equation*}
and $\Psi_0=\ker(\psi)$.
\end{definition}
We point out that as a particular case of the Separating property we may choose $u_0=u_-=0$ and $u_+\neq 0$ for instance. Several examples of functions $\sigma$ that fulfill the Separating property are provided in \cite{benth2023neural}. Observe that by construction a function $\sigma:\X\to\X$ satisfying this property cannot be constant: compare again \cite{hornik1991approximation}.

The following first main result shows that $\mathfrak N(\sigma)$ is dense in $L^p(\mu),1\le p<\infty$ if the activation function $\sigma$ satisfies the separating property and it is von Neumann-bounded. The result is a combination of classical techniques from \cite{hornik1991approximation}, where a similar result has been shown for the case $\X = \mathbb{R}^n$, and more recent techniques from our paper \cite{benth2023neural}, where a universal approximation theorem on a Fr\'echet space $\X$ and with respect to the topology of uniform convergence on compacts has been proved.   

\begin{theorem}\label{thm: main}
    Let $\X$ be a real Fr\'echet space, and let $\sigma:\X\to\X$ be Borel measurable, von Neumann-bounded and satisfying the separating property. Let $\mu$ be a Radon probability measure on $\X$, and let $1\le p <\infty$. Then $ \mathfrak N(\sigma)$ is dense in $L^p(\mu)$. In other words, given $f\in L^p(\mu)$, then for any $\varepsilon>0$, there exists $ \sum_{m=1}^M \alpha_m\mathcal{N}_{\ell_m,A_m,b_m}\in  \mathfrak N(\sigma)$ with suitable $\alpha_m\in\mathbb R, \ell_m\in\X',A_m\in\mathcal{L}(\X)$ and $b_m\in\X$ such that
\begin{equation*}
    \norm{\sum_{m=1}^M \alpha_m\mathcal{N}_{\ell_m,A_m,b_m} - f}_{L^p(\mu)}^p=
    \int_\X \abs{
    \sum_{m=1}^M \alpha_m\mathcal{N}_{\ell_m,A_m,b_m}(x) - f(x) 
    }^p\mu(dx) <\varepsilon^p.
\end{equation*}
\end{theorem}

\begin{proof}

The proof will be carried out by contradiction. Namely, assume that $\text{cl}(\mathfrak N(\sigma)) \subsetneq L^p(\mu)$: clearly, $\text{cl}(\mathfrak N(\sigma))$ is clearly still a vector subspace.

We choose $v\in L^p(\mu)\setminus \text{cl}(\mathfrak N(\sigma))$. By the Hahn-Banach theorem we can find a linear and continuous functional $\phi: L^p(\mu)\to \R$ such that
\begin{equation*}
    \phi\big|_{\operatorname{cl}(\mathfrak N(\sigma))} = 0, \quad \phi(v) =1,
\end{equation*}
and so $\phi$ cannot be identically zero. The Riesz representation theorem for $L^p$ spaces (refer to e.g. \cite[p. 375]{conway2010}: observe that to handle the case $p=1$ we have available that $\mu$ is $\sigma$-finite (clearly)), there exists $g\in L^{p'}(\mu)\setminus \{0\}$ with $p'=\frac{p}{p-1}\in (1,\infty]$ such that
\[
\phi(u)  = \int_\X u(x) g(x)\,\mu(dx), \quad u\in L^p(\mu),
\]
and with
\[
\int_\X u(x) g(x)\,\mu(dx) = 0, \quad u\in \text{cl}(\mathfrak N(\sigma)).
\]
This implies in particular that  for any $\ell\in\X',A\in\mathcal L(\X),b\in\X$ and $\lambda>0$ it holds
\begin{equation*}
    \int_\X \langle \ell,\sigma(\lambda(Ax+b))\rangle\, g(x)\,\mu(dx) = 0.
\end{equation*}

Observe that, as $\lambda\to \infty$, pointwise in $x\in\X$,
\begin{equation*}
     \langle \ell,\sigma(\lambda (Ax+b))\rangle \to
     \begin{cases}
     \langle \ell, u_+\rangle, \text{ if } Ax+b \in \Psi_+\\
      \langle \ell, u_-\rangle, \text{ if } Ax+b \in \Psi_-\\
       \langle \ell, u_0\rangle, \text{ if } Ax+b \in \Psi_0,\\
     \end{cases}
\end{equation*}
and note that for this property to hold the continuity of $\sigma$ is not necessary.

By the boundedness of the neurons (compare \eqref{eq: neurons are bounded}), we can find a constant $ C(\ell,\sigma)$ which does not depend neither on $\lambda$ nor on $x\in \X$ such that
\begin{equation*}
    \abs{\langle \ell,\sigma(\lambda (Ax+b))\rangle} \leq C(\ell,\sigma),
\end{equation*}
and thus
\[
\abs{\langle \ell,\sigma(\lambda (Ax+b))\rangle g(x)} \leq C(\ell,\sigma) g(x) 
\]
which belongs to $L^1(\mu)$. Therefore, the Dominated Convergence Theorem, upon sending $\lambda\to\infty$, gives us
\begin{equation*}
    \begin{split}
    &\lim_{\lambda\to \infty}\int_\X \langle \ell,\sigma(\lambda(Ax+b))\rangle\, g(x)\,\mu(dx) \\
    &\quad=  \langle \ell, u_+\rangle \int_{A^{-1}(\Psi_+-b)} g(x)\,\mu(dx) +
    \langle \ell, u_-\rangle \int_{A^{-1}(\Psi_--b)} g(x)\,\mu(dx) \\
    &\quad\quad+
    \langle \ell, u_0\rangle \int_{A^{-1}(\Psi_0-b)} g(x)\,\mu(dx)
    \end{split}
\end{equation*}
for any $\ell\in\X',A\in\mathcal L(\X)$ and $b\in\X$.

Let us first assume that $u_+ \notin \Span \{u_0,u_-\}$. Then by the Hahn-Banach theorem for Fr\'echet spaces (see e.g. Conway \cite[Chap IV, Cor. 3.15]{conway2010}) we can choose $\ell\in\mathfrak X'$ such that $\langle \ell, u_+\rangle=1$ and  $\langle \ell, u_-\rangle=\langle \ell, u_0\rangle=0$. This gives us
\[
\int_{A^{-1}(\Psi_+-b)} g(x)\,\mu(dx) = 0
\]
for any $b\in\X$ and $A\in\mathcal L(\X)$. We fix now $s\in\R$ and $b\in\X$ such that $s=\psi(-b)$. Then, it holds that $$\Psi_+ -b=\psi^{-1}(s,\infty)$$ and thus $$\int_{A^{-1}\circ \psi^{-1}(s,\infty)} g(x)\,\mu(dx)=0$$
for each $s\in\R$ and $A\in\mathcal L(\X)$. By Lemma \ref{lemma: rotation of hyperplanes}, we obtain that for any linear and continuous functional $\gamma\in\X'$ and scalar $s\in\R$
\begin{equation}\label{eqn:right:interval}
\int_{\{\gamma>s\}} g(x)\,\mu(dx)=0.
\end{equation}
In the case $u_- \notin \Span \{u_0,u_+ \}$ instead, by another application of Hahn-Banach theorem we obtain now that
\begin{equation}\label{eqn:left:interval}
\int_{\{\gamma < s\}} g(x)\,\mu(dx)=0, \quad \gamma\in\X',s\in \R.
\end{equation}

We now define the non-zero finite signed Borel measure
\[  
\nu(B):=\int_Bg(x)\,\mu(dx),\quad B\in\mathcal B(\X).
\]
Clearly, it holds
\[
\nu_+ = g_+\mu,\quad \nu_-=g_-\mu,
\]
where $\nu_+$ and $\nu_-$ are the positive part and negative part respectively of the Hahn-Jordan decomposition of $\nu$, and $g_+$ and $g_-$ are respectively the positive and negative part of $g$. Observe, that given $\varepsilon>0$ and $B\in\mathcal{B}(\X)$, since $\mu$ is Radon, we may find $K_\varepsilon\subset B$ compact such that
\[
\mu(B\setminus K_\varepsilon) < \left[\frac{\varepsilon}{\norm{g_+}_{L^{p'}(\mu)} +  \norm{g_-}_{L^{p'}(\mu)}}\right]^p.
\]
Hence,
\[
\abs{\nu}(B\setminus K_\varepsilon) =\int_{B\setminus K_\varepsilon}g_+(x)\,\mu(dx) + \int_{B\setminus K_\varepsilon}g_-(x)\,\mu(dx) < \varepsilon
\]
where $\abs{\nu}$ is the total variation of $\nu$. Namely, we have showed that $\nu$ is Radon as well. 

From either \eqref{eqn:right:interval} or \eqref{eqn:left:interval} we deduce that
\[
\nu_+(\gamma\in V) = \nu_-(\gamma\in V)
\]
for any $\gamma\in \X'$ and $V\in\mathcal B(\R)$, and so
\[
\hat{\nu}_+ = \hat{\nu}_-,
\]
where $\hat{}$ denotes the Fourier transform of a measure: see Definition \ref{def: Fourier}. From Lemma \ref{lemma: Fourier}, it follows that $\nu_+$ and $\nu_+$ must coincide on $\Sigma(\X')$, the sigma-algebra generated by all linear and continuous functional on $\X$. However, always from Lemma \ref{lemma: Fourier}, we conclude that $\mu_+$ and $\mu_-$ coincide also on the larger sigma-algebra $\mathcal B(\X)$ because we have seen that they are both Radon. Thus, $\nu\equiv 0$, namely $g\equiv 0$, which is a contradiction. We conclude that the $\mathfrak N(\sigma)$ is dense in $L^p(\mu)$.

\end{proof}

\begin{remark}
    It is known that Radon probability measures on a Fr\'echet space $\X$ are concentrated on separable reflexive Banach spaces. Namely, for any $\mu$ Radon there exists a linear subspace $E\subset\X$ (which depends on $\mu$) such that $\mu(E)=1$ and $E$ can be endowed with a norm making it a separable reflexive Banach space: refer to \cite{bogachev2007measure}. Therefore, one in principle could work in the Banach space category and use neural architectures based on activation functions $\sigma:E\to E$. The drawback of this approach is that, since $E$ depends on the Radon measure $\mu$, we would lose some flexibility when choosing a different Radon measure $\mu'$, in comparison to the ``universal'' approach we are adopting here, i.e. $\sigma:\X\to\X$ independent of $\mu$.   
\end{remark}

\section{Finite-dimensional approximation}\label{sec: Finite-dimensional approximation}

Similarly in the spirit of \cite{benth2023neural}, we now face the issue of approximating the $L^p$ neural networks introduced in the previous section and which are apriori specified by an infinite-number of parameters, via suitable architectures which are instead determined by a finite number of parameters only. Evidently, this can only work if we can approximate any given $x\in \X$ sufficiently well with a finite dimensional quantity as otherwise we could not even represent $x$ in a computer. We therefore need spaces granting an ``approximation of the identity'' property. Therefore, similarly, to the ideas used in \cite{benth2023neural}, from now on we impose that the underlying space $\X$ is a separable Banach space with norm $\norm{\cdot}$ that admits a Schauder basis $\X$. This quite strong requirement is sufficient for the moment, and we will show in later sections how it can be removed in some cases.

More precisely, we assume that $\X$ admits a normalized Schauder basis $(e_k)_{k\in\N}$, namely each $x\in\X$ has a unique representation $x=\sum_{k=1}^\infty x_k e_k,x_k\in\R$ and $\norm{e_k}=1$ for all $k$.
It follows as in Schaefer \cite[Thm. 9.6, p. 115]{schaefer1971topological} that 
\begin{equation*}
    \Pi_N: \X \to \Span\{e_1,\dots ,e_N\},\quad x\mapsto \sum_{k=1}^N x_k e_k,\quad N\in\N
\end{equation*}
is linear and bounded with $\sup_{N\in\N}\norm{\Pi_N}_{op}\leq C$ for some suitable constant $C\geq 1$, and that for any $K\subset\X$ compact we have $\sup_{x\in K}\norm{x-\Pi_N x}\to 0$ as $N\to\infty$, and so in particular we have pointwise convergence of $\Pi_Nx$ to $x$ for any $x\in\X$. $\norm{\cdot}_{op}$ denotes the operator norm of bounded linear operator. We also introduce for convenience the canonical linear continuous projectors $\beta_k,k\in\N$ associated to the basis, namely
\[
\beta_k:\X\to\R,\quad x\mapsto x_k,\quad k\in\N.
\]
In this way, $x=\sum_{k=1}^\infty\langle\beta_k,x\rangle e_k,\,x\in\X$.

Since we are interested here in $L^p$ topology rather than the topology of uniform convergence on compacts, in contrast to \cite{benth2023neural},
we only need to assume that the activation function $\sigma:\X\to \X$ is continuous (but not necessarily Lipschitz). We also note that von Neumann-boundedness here simply means
\[
\sup_{x\in\X}\norm{\sigma(x)} < \infty,
\]
i.e. $\sigma$ is bounded. 

We then have:

\begin{proposition}\label{prop: finite dimensional approx, Banach}
Let $(\X,\norm{\cdot})$ be a real separable Banach space that admits a normalized Schauder basis $(e_k)_{k\in\N}$ with projections $\Pi_N,N\in\N$. Let $\sigma$ be continuous and bounded. Let $\mu$ be a probability measure (not necessarily Radon) on $\X$ and let $1\le  p <\infty$. Let $f\in L^p(\mu)$ and $\varepsilon>0$. Assume there exists $\mathcal N^\epsilon\in \mathfrak N(\sigma) \subset L^p(\mu)$
\begin{equation*}
    \mathcal N^{\epsilon} (x) = \sum_{j=1}^M\langle \ell_j,\sigma(A_jx+b_j)\rangle,\quad x\in\X
\end{equation*}
with suitable $M\in \N,\ell_j\in\X',A_j\in\mathcal{L}(\X)$ and $b_j\in\X$ such that 
\begin{equation*}
    \norm{f-\mathcal N^{\epsilon}}_{L^p{(\mu)}}<\varepsilon.
\end{equation*}
Fix $\delta>0$. Then there exists $N_\ast=N_\ast(\mathcal N^\epsilon,\delta)\in\N$ such that for $N\geq N_\ast$
\begin{equation}\label{approx:finite}
\norm{f-\sum_{j=1}^M\langle \ell_j\circ\Pi_N,\sigma(
    \Pi_{N}A_j\Pi_{N}\cdot+\Pi_{N}b_j)\rangle}_{L^p(\mu)} < \varepsilon+\delta.
\end{equation}
\end{proposition}

\begin{proof}
Fix $j=1,\dots, M$. By the pointwise convergence of the projections $\Pi_N$, we have that $\Pi_N A_j \Pi_Nx\to A_jx$ for each $x\in\X$ as $N\to\infty$, because
\[
\begin{split}
    \norm{\Pi_N A_j \Pi_Nx - A_jx} & \le \norm{\Pi_NA_j\Pi_Nx-\Pi_NA_jx} + \norm{\Pi_NA_jx-A_jx}\\
    & \le C \norm{A_j\Pi_Nx-A_jx} + \norm{\Pi_NA_jx-A_jx}\\
    & \le C \norm{A_j}_{op} \norm{\Pi_Nx-x} + \norm{\Pi_NA_jx-A_jx}\to 0
\end{split}
\]
because $\sup_{N\in\N}\norm{\Pi_N}_{op}\leq C<\infty$. Since the activation function is now assumed to be continuous, clearly
\[
\sigma(\Pi_N A_j \Pi_Nx+\Pi_Nb_j) \to \sigma(A_jx+b_j), \quad x\in \X
\]
as well. Besides, if we have an arbitrary sequence $(y_N)_N\subset\X$ such that $y_N\to y\in\X$, then
\[
\begin{split}
    \norm{\Pi_Ny_N-y} & \le \norm{\Pi_Ny_N-\Pi_Ny} + \norm{{\Pi_Ny-y}}\\
     & \le C\norm{ y_N-y} + \norm{{\Pi_Ny-y}} \to 0.
\end{split}
\]
Thus, by setting $y_N:=\sigma(\Pi_N A_j \Pi_Nx+\Pi_Nb_j)$ and $y:=\sigma(A_jx+b_j)$ in the previous equation, we infer
\[
\Pi_N\sigma(\Pi_N A_j \Pi_Nx+\Pi_Nb_j) \to \sigma(A_jx+b_j),\quad x\in\X
\]
and therefore 
\[
\langle\ell_j, \Pi_N\sigma(\Pi_N A_j \Pi_Nx+\Pi_Nb_j)\rangle \to \langle \ell_j,\sigma(A_jx+b_j)\rangle,\quad x\in\X,
\]
namely, as $N\to\infty$
\[
\sum_{j=1}^M\langle \ell_j\circ\Pi_N,\sigma(
    \Pi_{N}A_j\Pi_{N}x+\Pi_{N}b_j)\rangle \to \sum_{j=1}^M\langle \ell_j,\sigma(A_jx+b_j)\rangle
\]
pointwise in $x\in\X$.

Besides, for any $x\in\X$,
\[
\begin{split}
    \abs{\sum_{j=1}^M\langle \ell_j\circ\Pi_N,\sigma(
    \Pi_{N}A_j\Pi_{N}x+\Pi_{N}b_j)\rangle}&\le \sum_{j=1}^M \abs{\langle \ell_j\circ\Pi_N,\sigma(
    \Pi_{N}A_j\Pi_{N}x+\Pi_{N}b_j)\rangle}\\
    &\le \sum_{j=1}^M\norm{\ell_j}_\ast\norm{\Pi_N\sigma(
    \Pi_{N}A_j\Pi_{N}x+\Pi_{N}b_j)}\\
    &\le C\sum_{j=1}^M\norm{\ell_j}_\ast\norm{\sigma(
    \Pi_{N}A_j\Pi_{N}x+\Pi_{N}b_j)}\\
    &\le C(\sigma,\ell_1,\dots,\ell_M)
\end{split}
\]
where $C(\sigma,\ell_1,\dots,\ell_M)$ is a constant that does not depend on $N$. Here $\norm{\cdot}_\ast$ denotes the dual norm of $\X'$. By the Dominated Convergence Theorem, we have 
\[
\sum_{j=1}^M\langle \ell_j\circ\Pi_N,\sigma(
    \Pi_{N}A_j\Pi_{N}\cdot+\Pi_{N}b_j)\rangle \to 
    \sum_{j=1}^M\langle \ell_j,\sigma(A_j\cdot+b_j)\rangle\quad \text{in }  L^p(\mu).
\]
Then for any $\delta>0$ there exists $N_\ast=N_\ast(\mathcal N^\epsilon,\delta)\in\N$ such that for $N\geq N_\ast$
\[
\norm{\sum_{j=1}^M\langle \ell_j\circ\Pi_N,\sigma(
    \Pi_{N}A_j\Pi_{N}\cdot+\Pi_{N}b_j)\rangle -
    \sum_{j=1}^M\langle \ell_j,\sigma(A_j\cdot+b_j)\rangle}_{L^p(\mu)}<\delta
\]
and hence the thesis.
    
\end{proof}

The following result is evident:
\begin{proposition}\label{prop: finite dimensional approx, Banach 2}
    Let $(\X,\norm{\cdot})$ be a real separable Banach space that admits a normalized Schauder basis $(e_k)_{k\in\N}$ with projections $\Pi_N,N\in\N$. Let $\sigma$ be continuous, bounded and separating. Let $\mu$ be a Radon probability measure on $\X$ and let $1\le  p <\infty$. Then the space of neural architectures
    \[
    \left\{\sum_{j=1}^M\langle \ell_j\circ\Pi_N,\sigma(
    \Pi_{N}A_j\Pi_{N}\cdot+\Pi_{N}b_j)\rangle ;\; M,N\in\N,\ell_j\in\X',A_j\in\mathcal{L}(\X),b_j\in\X\right\}
    \]
    is dense in $L^p(\mu)$.
\end{proposition}

Let us briefly comment on the last result. Despite the fact that the neural network 
\[
\X\ni x\mapsto \sum_{j=1}^M \langle
    \ell_j\circ\Pi_N,\sigma(\Pi_N A_j\Pi_Nx + \Pi_Nb_j)\rangle
\]
can process infinite-dimensional inputs $x\in\X$, its ``weights'' $\ell_j\circ\Pi_N,\Pi_N\circ A_j\circ\Pi_N$ and $\Pi_N\circ b_j$ are specified by a finite number of parameters. More precisely, the action of $\ell_j$ will be prescribed by the scalars $\ell_j(e_1),\dots,\ell_j(e_N)$, the action of $\Pi_N\circ A_J\circ\Pi_N$ will be specified by $\left\{ (A_je_m,e_k)_V \right\}_{m,k=1}^N$, and $\Pi_Nb_j$ is specified by an $N$ dimensional vector. This means that the neural network can be implemented in a machine possessing only finite memory and finite computing power. Its architecture is similar to the architecture of a classical 1-layer feed forward neural network: however, in the present setting, the function $\Pi_N \circ \sigma$ restricted to $\Span\{e_1,\dots ,e_N\}$ is multidimensional.

\section{Approximation results for more general domain}\label{sec: qp}
In this section, we leverage the results from the previous sections, and we will extend the scope of the neural architectures introduced above, by proving a more general universal approximation theorem in the $L^p$ sense. 

The classes of input spaces and neural networks we are going to use have been introduced in \cite{galimberti2024neural}: these spaces are called quasi-Polish spaces. We briefly recall here the main properties we will need: for more properties on quasi-Polish spaces and plenty of examples thereof, we refer again to \cite{galimberti2024neural}.

Let us consider the separable Hilbert space of square integrable sequences $V=\ell^2(\N)$. Define
\[
\mathcal{Q}=\{ a\in V; \, 0\leq a_i\leq 1/i, \,i\in\N
\}.
\]
It can be easily shown that $\mathcal{Q}$ is a compact subset of $V$.

\begin{definition}
    A topological space $(\X,\tau)$ is called quasi-Polish if there exists  a countable family $\{h_i:\X\to [0,1/i]\}_{i=1}^\infty$ of $\tau$-continuous functions which separate points of $\X$, namely for each $x_1,x_2\in\X$ with $x_1\neq x_2$ there exists $i\in \N$ such that
\[
h_i(x_1) \neq h_i(x_2).
\]
Such a family will be called a separating sequence.
\end{definition}

The assumption is very simple, and it is satisfied by a huge class of topological spaces. For instance all Polish spaces fall into this category, i.e. Polish spaces are quasi-Polish. It immediately gives rise to the following consequences:
\begin{enumerate}
    \item The induced map
    \[
    \X \ni x \stackrel{H}{\longmapsto} H(x) = (h_1(x),h_2(x),\dots) \in  \mathcal{Q} \subset V 
    \]
    is 1-1 and continuous, but in general it is not a homeomorphism of $\X$ onto a subspace of $ \mathcal Q $, i.e. in general it fails to be an embedding.
    \item $H$ defines another topology $\tau_H$ on $\X$ which is weaker than $\tau$, i.e. $\tau \supset \tau_H$. This last topology is metrizable though, and hence both $\tau_H$ and $\tau$ are Hausdorff.
    \item By the well-known minimal property of compact topologies, both topologies coincide on $\tau$-compact sets $\mathcal{K}\subset\X$, and hence $\tau$-compact sets are metrizable.
    \item For any $\mathcal{K}\subset \X$ $\tau$-compact
    \[
        H\big |_{\mathcal{K}}:\mathcal{K}\to H(\mathcal{K})
    \]
    is a homeomorphism. Therefore, $H(\mathcal{K})$ is compact in $\mathcal Q$. 
    Besides, $\mathcal{K}$ is compact if and only if it is sequentially compact.
    \item For all $E\subset\X$ $\sigma$-compact subspaces of $(\X,\tau)$
    \[
        H\big |_{E}:E\to H(E)
    \]
    is a measurable isomorphism.
\end{enumerate}

This last property will be very important for our goals, because by means of a change of variables we will be able to ``transport'' the results from Section \ref{main:approx} in the present setting. So, in the rest of the section, we are going to assume that $(\X,\tau)$ is quasi-Polish with some separating sequence $(h_i)_{i=1}^\infty$ and that $(\X,\tau)$ is $\sigma$-compact.

Before proving our next result, we need some preliminary observations: first of all, being $\X$ $\sigma$-compact, it can be written as
\[
\X =\bigcup_{n=1}^\infty K_n,\quad K_n \text{ compact}
\]
and so $H(\X)=\bigcup_{n=1}^\infty H(K_n)$ is also $\sigma$-compact by continuity of $H$, and therefore $H(\X)\in \mathcal B(V)$. Besides, $\mathcal{B}(H(X))=H(X)\Cap \mathcal{B(V)}:=\{H(\X)\cap B;B\in \mathcal B(V)\}$. Suppose $\mu$ is a Radon probability measure on $\X$. Then the pushforward probability measure
\[
\mu\circ H^{-1} :\mathcal{B}(V) \to [0,1]
\]
is also Radon: indeed, take an arbitrary $B\in\mathcal B(V)$ and $\varepsilon>0$. Evidently, without loss of generality, we can assume $B\subset H(\X)$. Set $B':=H^{-1}(B)\in\mathcal{B}(\X)$. Then there exists $K'_\varepsilon\subset B'$ compact such that $\mu(B'\setminus K'_\varepsilon)<\varepsilon$. Observe that $H(K'_\varepsilon)$ is compact, as $H$ is continuous, and trivially $H(K'_\varepsilon)\subset B$. Therefore,
\[
\mu\circ H^{-1}(B\setminus H(K'_\varepsilon)) = \mu(H^{-1}(H(B')\setminus H(K'_\varepsilon))) = \mu(B'\setminus K'_\varepsilon)<\varepsilon,
\]
and hence the claim. 

Moreover, given a function $g$ in $L^p(\mu),1\le p<\infty$, upon extending $g\circ H^{-1}$ to 0 outside $H(\X)$, we end up with a Borel measurable function on $V$ such that
\[
\begin{split}
    \int_V \abs{g\circ H^{-1}(v)}^p\,\mu\circ H^{-1}(dv) &= \int_{H(\X)} \abs{g\circ H^{-1}(v)}^p\,\mu\circ H^{-1}(dv)\\
    &=\int_\X \abs{(g\circ H^{-1})\circ H(x)}^p\,\mu  (dx)\\
    & =  \int_\X \abs{g(x)}^p\,\mu  (dx)
\end{split}
\]
by a change of variable, i.e. $g\circ H^{-1}\in L^p(\mu\circ H^{-1})$: in the sequel, we will always tacitly use this extension.

Let us give an important example of quasi-Polish spaces which are $\sigma$-compact.
\begin{example}\label{ex: reflexive separable Banach}
    Let $(\X,\norm{\cdot})$ be a separable reflexive Banach space. Let us endow $\X$ with its weak topology $\tau(\X,\X')$, i.e. the coarsest topology on $\X$ which makes all the elements of $\X'$ continuous. Then $(\X,\tau(\X,\X'))$ is quasi-Polish, as showed in \cite[Example 2.5]{galimberti2024neural}. Besides, since $\X$ as a Banach space is reflexive, by Kakutani's Theorem, we immediately deduce that $(\X,\tau(\X,\X'))$ is $\sigma$-compact. Besides, the following identity also holds
    \[
    \mathcal{B}(\tau(\X,\X')) = \Sigma(\X')=\mathcal{B}(\X)
    \]
    where $\Sigma(\X')$ is the $\sigma$-algebra generated by the element of $\X'$ and $\mathcal{B}(\X)$ is the Borel $\sigma$-algebra generated by the strong topology. In other words, even though $(\X,\norm{\cdot})$ and $(\X,\tau(\X,\X'))$ are topologically different, they are equivalent as measurable spaces.
    
\end{example}

We are ready to state and prove the main result of this section:

\begin{proposition}\label{prop: qp approx}
    Let $(\X,\tau)$ be a quasi-Polish space with induced map $H$, and assume that $\X$ is $\sigma$-compact. Let $1\le p <\infty$ and $\mu$ a Radon probability measure on $\X$. Assume to have an activation function $\sigma: V\to V$ Borel measurable, with bounded range and satisfying the separating property. Let $f\in L^p(\mu)$ and $\varepsilon>0$. Then there exists $ \sum_{m=1}^M \alpha_m\mathcal{N}_{\ell_m,A_m,b_m}\in  \mathfrak N(\sigma)$ with suitable $\alpha_m\in\mathbb R, \ell_m\in V',A_m\in\mathcal{L}(V)$ and $b_m\in V$ such that
\begin{equation*}
    \norm{\sum_{m=1}^M \alpha_m\mathcal{N}_{\ell_m,A_m,b_m}\circ H - f}_{L^p(\mu)}^p=
    \int_\X \abs{
    \sum_{m=1}^M \alpha_m\mathcal{N}_{\ell_m,A_m,b_m}(H(x)) - f(x) 
    }^p\mu(dx) <\varepsilon^p.
\end{equation*}  
\end{proposition}
\begin{proof}
    The proof follows straightforwardly from the discussion above: given $f\in L^p(\mu)$, we may consider $f\circ H^{-1}\in L^p(\mu\circ H^{-1})$ (trivially extended outside $H(\X)$). Since $\mu\circ H^{-1}$ is Radon on $V$, we can apply Theorem \ref{thm: main}, and find a neural network on $V$
    \[
    \mathcal{NN}=\sum_{m=1}^M \alpha_m\mathcal{N}_{\ell_m,A_m,b_m}\in  \mathfrak N(\sigma)    
    \]
 with suitable $\alpha_m\in\mathbb R, \ell_m\in V',A_m\in\mathcal{L}(V)$ and $b_m\in V$ such that
 \[
 \int_V \abs{f\circ H^{-1}(v) - \mathcal{NN}(v)}^p\mu\circ H^{-1}(dv) =  \int_{H(\X)} \abs{f\circ H^{-1}(v) - \mathcal{NN}(v)}^p\mu\circ H^{-1}(dv)<\varepsilon^p.
 \]
Hence, by a change of variable once more
\[
\int_\X\abs{f(x)-\mathcal{NN}\circ H(x)}^p\mu(dx) = 
\int_{H(\X)} \abs{f\circ H^{-1}(v) - \mathcal{NN}(v)}^p\mu\circ H^{-1}(dv)<\varepsilon^p.
\]
   
\end{proof}

The next result is an application of the results in Section \ref{sec: Finite-dimensional approximation}:

\begin{proposition}\label{prop: qp approx finite-dim}
    Assume the same setting of Proposition \ref{prop: qp approx} and in addition that the activation function $\sigma:V\to V$ is continuous. Fix $f\in L^p(\mu)$. Then for any $\varepsilon>0$ there exists  $ \sum_{m=1}^M \alpha_m\mathcal{N}_{\ell_m,A_m,b_m}\in  \mathfrak N(\sigma)$ with suitable $\alpha_m\in\mathbb R, \ell_m\in V',A_m\in\mathcal{L}(V)$ and $b_m\in V$, and an integer $N_\varepsilon$ such that for all integers $N\ge N_\varepsilon$ it holds
\begin{equation*}
    \int_\X \abs{
    \sum_{m=1}^M \alpha_m \langle
    \ell_m\circ\Pi_N,\sigma(\Pi_N A_m\Pi_NH(x) +\Pi_Nb_m)\rangle - f(x) 
    }^p\mu(dx) <\varepsilon^p.
    \end{equation*}
where $\{\Pi_N\}_N$ is a sequence of orthogonal projections with respect to an orthonormal basis $(e_k)_{k\in\N}$ of $V$.  
\end{proposition}

\begin{proof}
    The result follows by combining Propositions \ref{prop: qp approx} and \ref{prop: finite dimensional approx, Banach}.
\end{proof}

\begin{remark}
We observe that, by choosing the canonical orthonormal basis of $V$, namely $$e_k:=(0,0,\dots, 0,1,0,\dots),\quad k\in \N$$ with entry equal to 1 at the $k$-th slot, then the input $H(x)\in V$ becomes
\[
\Pi_N H(x) = h_1(x)e_1 + \dots + h_N(x)e_N = (h_1(x),\dots,h_N(x),0,0,\dots). 
\]
Therefore, we see that specifying the whole separating sequence $(h_i)_{i=1}^\infty$ is not even necessary, but that is enough to stop at $N\in\N$ sufficiently large. In light of this, let us revisit Example \ref{ex: reflexive separable Banach}: we are given a reflexive separable Banach space $(\X,\norm{\cdot})$ and we want to approximate functions $f\in L^p(\mu)$ with $1\le p<\infty$ and $\mu$ a Radon measure on $\X$. Certainly one can accomplish this by means of Theorem \ref{thm: main}, even though the ensuing neural networks are specified apriori by an infinite number of parameters. This issue could be circumvented by means of Proposition \ref{prop: finite dimensional approx, Banach 2}, provided that the space $\X$ grants the existence of a Schauder basis, a property that is not always true. However, if we endow $\X$ with its weak topology $\tau(\X,\X')$, we have seen that it becomes a quasi-Polish space which, as a measurable space, is equivalent to the original $(\X,\norm{\cdot})$. Therefore, our original problem of approximating functions $f\in L^p(\mu)$ via neural networks specified by a finite number of parameters can now be addressed by means of 
Proposition \ref{prop: qp approx finite-dim}.
    
\end{remark}

\section{Approximation results for more general codomain}\label{sec: general codomain}

In this section we are going to show that our results can be extended to vector valued function $f=f(x)\in E$, where $(E,\norm{\cdot}_E)$ is a Banach space endowed with a normalised Schauder basis $(s_k)_{k\in\N}$ (and hence separable). As recalled at the beginning of Section \ref{sec: Finite-dimensional approximation}, $\beta^E_k,\,k\in\N$ will denote here the canonical linear projectors associated to the basis, while $\Pi_N^E,\,N\in\N$ will denote the associated projection operators. 

Recall, that given an arbitrary finite measure space $(\Omega,\mathcal{A},\mu)$ and $1\le p <\infty$, the $p$-Bochner space $L^p(\mu;E)$ is defined as
\[
L^p(\mu;E):=\left\{f:\Omega\to E; \, f \text{ is } \mathcal A / \mathcal{B}(E) \text{ measurable and } \int_{\Omega} \norm{f(\omega)}_E^p\mu(d\omega)<\infty\right\},
\]
where as usual we identify functions that agree $\mu$-a.e. It is a Banach space with respect to the norm
\[
\norm{f}_{L^p(\mu;E)}:= \left[\int_{\Omega} \norm{f(\omega)}_E^p\mu(d\omega)\right]^{1/p},\quad f\in L^p(\mu;E).
\]

We start with an easy lemma

\begin{lemma}\label{lemma: Lp convergence of projections}
    Let $1\le p <\infty$ and $(\Omega,\mathcal{A},\mu)$ be an arbitrary finite measure space. Let $f\in L^p(\mu;E)$. Define
    \[
    (\Pi_N^Ef)(\omega):= \Pi_N^Ef(\omega),\quad \omega\in\Omega.
    \]    
    Then $\Pi_N^Ef\in L^p(\mu;E)$ and 
    \[
    \Pi_N^Ef\to f \quad \text{in } L^p(\mu;E)\text{ as }N\to\infty.
    \]    
\end{lemma}
\begin{proof}
Given $f\in L^p(\mu;E)$, since the projection operators $\Pi^E_N$ are by assumption linear and continuous maps, it follows that $\Pi^E_Nf=\Pi_N^E\circ f$ is $\mathcal{A}/\mathcal{B(E)}$ measurable. Besides,
\[
\int_\Omega \norm{(\Pi^E_Nf)(\omega)}^p_E\mu(d\omega) \le C \int_\Omega \norm{ f(\omega)}^p_E\mu(d\omega),
\]
where $C:= \sup_N{\norm{\Pi^E_N}}_{op}<\infty$, namely $\Pi_N^Ef\in L^p(\mu;E)$. Since pointwise it holds good that $(\Pi^E_Nf)\to f$ as $N\to\infty$, we have
\[
\norm{(\Pi^E_Nf)(\omega)-f(\omega)}_E^p \to 0
\]
and
\[
\norm{(\Pi^E_Nf)(\omega)-f(\omega)}_E^p \le 2^{p-1}\left( C \norm{f(\omega)}^p_E + \norm{f(\omega)}^p_E\right) \in L^1(\mu).
\]
By the Dominated Convergence Theorem we obtain
\[
\int_\Omega \norm{(\Pi^E_Nf)(\omega)-f(\omega)}_E^p\mu(d\omega)\to 0\quad\text{as } N\to\infty.
\]
    
\end{proof}

After this preparation, we are ready to state and prove the main result of this section: 

\begin{proposition}\label{prop: vectorial case}
    Let $\X$ be a real Fr\'echet space, and let $\sigma:\X\to\X$ be Borel measurable, 
    von Neumann-bounded and satisfying the separating property. Let $(E,\norm{\cdot}_E)$ be a Banach space endowed with a normalised Schauder basis $(s_k)_{k\in\N}$.  Let $\mu$ be a Radon probability measure on $\X$, and let $1\le p <\infty$. Let $f\in L^p(\mu;E)$. Then for any $\varepsilon>0$ there exist $N_\varepsilon\in\N$ and $\mathcal{N}_1,\dots,\mathcal{N}_{N_\varepsilon}\in\mathfrak{N}(\sigma)$ such that, by setting
    \[
    \mathcal{N}(x):=\sum_{n=1}^{N_\varepsilon}\mathcal{N}_n(x)s_n,\quad x\in\X
    \]
    we have
    \[
    \norm{\mathcal{N}-f}_{L^p(\mu;E)}<\varepsilon.
    \]    
\end{proposition}
\begin{proof}
    Given $f\in L^p(\mu;E)$ and $\varepsilon>0$, we know that from Lemma \ref{lemma: Lp convergence of projections} there exists $N_\varepsilon\in\N$ such that for all integers $N\ge N_\varepsilon$ it holds
    \[
    \norm{\Pi_N^Ef-f}_{L^p(\mu;E)} < \varepsilon/2.
    \]
    By construction,
    \[
    (\Pi^E_{N_\varepsilon}f)(x) = \sum_{n=1}^{N_\varepsilon}\langle\beta^E_n ,f(x)\rangle s_n.
    \]
    Since the linear projectors are continuous, it follows that $\X\ni x\mapsto \langle\beta^E_n ,f(x)\rangle\in\R $ is $\mathcal{B}(\X)/\mathcal{B}(\R)$ measurable, and
    \[
    \int_\X\abs{\langle\beta^E_n ,f(x)\rangle}^p\mu(dx) \le \norm{\beta^E_n}_{\ast} \int_\X\norm{f(x)}_E^p\mu(dx) <\infty,
    \]
    namely, $\langle \beta^E_n,f\rangle\in L^p(\mu),n\in\N$. 

    In virtue of Theorem \ref{thm: main}, there exist  $\mathcal{N}_1,\dots,\mathcal{N}_{N_\varepsilon}\in\mathfrak{N}(\sigma)$ such that
    \[
    \norm{\mathcal N_n-\langle\beta^E_n,f\rangle}_{L^p(\mu)} < \varepsilon/2N_\varepsilon,\quad n=1,\dots, N_\varepsilon.
    \]
    Let us set $\mathcal{N}(x):=\sum_{n=1}^{N_\varepsilon}\mathcal{N}_n(x)s_n, x\in\X$: evidently, for any $n=1,\dots, N_\varepsilon$, $\mathcal N_ns_n$ is $\mathcal{B(\X)}/\mathcal{B}(E)$ measurable, and 
    we have
    \[
    \int_\X\norm{\mathcal{N}_n(x)s_n}_E^p\mu(dx) = \int_\X\abs{\mathcal{N}_n(x)}_E^p\mu(dx) <\infty, 
    \]
    i.e. $\mathcal{N}_ns_n\in L^p(\mu;E)$ and thus $\mathcal{N}\in L^p(\mu;E)$ as well. Therefore, we conclude
    \[
    \begin{split}
        \norm{\mathcal{N}-f}_{L^p(\mu;E)} & \le \norm{\Pi^E_{N_\varepsilon}f-f}_{L^p(\mu;E)} + \norm{\Pi^E_{N_\varepsilon}f-\mathcal N}_{L^p(\mu;E)} \\
        &< \varepsilon/2 + \norm{\sum_{n=1}^{N_\varepsilon}\left[\mathcal{N}_n - \langle \beta^E_n,f\rangle  \right] s_n}_{L^p(\mu;E)}\\
        &\le \varepsilon/2 +\sum_{n=1}^{N_\varepsilon}\norm{\mathcal{N}_n - \langle \beta^E_n,f\rangle }_{L^p(\mu)}\\
        &< \varepsilon/2+\varepsilon/2 = \varepsilon.
    \end{split} 
    \]

\end{proof}

\begin{remark}
    Clearly, a more general result would hold for a separable Banach space $(E,\norm{\cdot}_E)$ not necessarily admitting a Schauder basis, by leveraging the density of simple functions in $L^p(\mu;E)$. Namely, for any $f\in L^p(\mu;E)$ and $\varepsilon>0$ we can always find
    \[
    \chi= \sum_{k=1}^Ky_kI_{A_k},\quad y_k\in E,A_k\in\mathcal{A},
    \]
    such that $\norm{f-\chi}_{L^p(\mu;E)}<\varepsilon/2$. Each indicator function $I_{A_k}\in L^p(\mu)$ can then be in turn approximated by a suitable neural network $\mathcal N_k\in\mathfrak N(\sigma)$ with $\norm{I_{A_k}-\mathcal N_k}_{L^p(\mu)}<\frac{\varepsilon}{2K\norm{y_k}_E}$ in virtue of Theorem \ref{thm: main}, and therefore the vector-valued neural network so defined
    \[
    \mathcal{N}(x):=\sum_{k=1}^{K}\mathcal{N}_k(x)y_k
    \]
    would guarantee that $\norm{f-\mathcal N}_{L^p(\mu;E)}<\varepsilon$. However, the price to pay in this setting is that, contrary to Proposition \ref{prop: vectorial case}, we do not have apriori a fixed sequence of vectors $(s_k)_{k\in\N}$ specifying our neural architectures, but rather a finite set of vectors $\{y_1,\dots,y_k\}$ which depends on the target function $f$ that we want to approximate.
\end{remark}

\begin{remark}
    If we further assume that $(\X,\norm{\cdot})$ is now a Banach space with normalized Schauder basis $(e_k)_{k\in\N}$, in light of this last result and Proposition \ref{prop: finite dimensional approx, Banach 2} it is evident how to obtain a neural architecture with values in $E$ specified by a finite number of parameters. If we write
    \[
    \mathcal{N}_n(x)=\sum_{j=1}^{J_n}\langle \ell_j^{(n)},\sigma(A^{(n)}_jx + b^{(n)}_j )\rangle,\quad n=1,\dots,N_\varepsilon
    \]
    for suitable $\ell^{(n)}_j\in\X',A^{(n)}_j\in\mathcal{L}(\X),b^{(n)}_j\in\X$, then my means of Proposition \ref{prop: finite dimensional approx, Banach 2} we can find for any $\delta>0$ an integer $N^\ast$ such that
    \[
    \mathcal{N}^\ast(x):=\sum_{j=1}^{N_\varepsilon}\mathcal{N}_n^\ast(x)s_n
    \]
    with
    \[
    \mathcal{N}_n^\ast(x)=\sum_{j=1}^{J_n}\langle \ell_j^{(n)}\circ \Pi_{N^\ast},\sigma(\Pi_{N^\ast} A^{(n)}_j\Pi_{N^\ast}x + b^{(n)}_j )\rangle,\quad n=1,\dots,N_\varepsilon
    \]
    satisfies 
    \[
    \norm{\mathcal{N}^\ast-f}_{L^p(\mu;E)}<\varepsilon+\delta.
    \]
    ($\Pi_N,N\in\N$ here are the projections associated to the Schauder basis $(e_k)_{k\in\N}$ of $\X$.)
\end{remark}

\section{An example: Denoising Diffusion Operators}\label{sec: diffusion}

Diffusion models have recently emerged as a powerful paradigm for generative modeling. They consist of two elements: a forward process that corrupts input data with Gaussian white noise, and a reverse process that learns a score function and which generates by denoising new samples resembling the original data. In the last years, they have enjoyed a spectacular success; however, they are mostly formulated on finite dimensional spaces, a fact that prevents them from being used in applications where the data has a functional form, as in geometric data analysis or scientific computing (e.g. numerical solutions of PDEs). Quite recently, some authors have started to consider diffusion models in infinite dimensional spaces, i.e. where the data, of which we want to generate new samples, lives e.g. in some function space: see \cite{pidstrigach2024infinite, lim2025score}. 

In broad strokes, we are given a separable Hilbert space $(H,\langle \cdot,\cdot\rangle)$  and a Borel probability measure $\mu$ on $H$, which is called the data measure. Let $u$ be the data, namely a $H$-valued random variable that distributes as $\mu$. The corruption process is carried out in this way: let $\mu_0$ be a Gaussian distribution on $H$ with 0 mean vector and covariance operator given by $C:H\to H$ linear, self-adjoint, non-negative and trace-class. Then we define
\[
v = u+ \eta
\]
where $\eta$ distributes as $\mu_0$ and it is independent of $\mu$. Let us call $\nu$ the distribution of $v$. By choosing as a reference measure the Gaussian $\mu_0$, under certain assumptions (refer to \cite{lim2025score} for details), one can prove that the Radon-Nikodym derivative of $\nu$ with respect to $\mu_0$ exists and it takes the form 
\[
\frac{d\nu}{d\mu_0}(w) = \exp(\Phi(w)),\quad w\in H
\]
for some Borel function $\Phi:H\to\R$. By setting $H_{\mu_0}:=C^{1/2}(H)$ and assuming that $\Phi$ is Fr\'echet differentiable along $H_{\mu_0}$, we can introduce the score of $\nu$ with respect to $\mu_0$ as
\[
D_{H_{\mu_0}} \Phi : H\to H_{\mu_0}',\quad D_{H_{\mu_0}} \Phi \equiv  D_{H_{\mu_0}}\log\frac{d\nu}{d\mu_0}.
\]
In order to carry out the denoising part, we need to deal with the score matching problem, which takes this form: given a class of neural networks architectures $G_\theta:H\to H_{\mu_0}'$, parametrized by $\theta\in\R^p$ for some $p$, the goal is  
\[
\min_{\theta\in\R^p} \mathfrak L(\theta)
\]
where the loss function is
\[
\mathfrak L(\theta):=\int_H\norm{G_\theta(x) - D_{H_{\mu_0}} \Phi(x) }^2_{H_{\mu_0}'}\nu(dx).
\]
We therefore see that the score matching problem fits into our $L^p$ framework, because one is required to approximate the unknown score $D_{H_{\mu_0}} \Phi : H\to H_{\mu_0}'$, a map between infinite dimensional space, in the $L^2$ sense. We also remark that, as in general $\operatorname{supp} \nu$ is not compact, the ``uniform-norm'' framework offered by \cite{benth2023neural} would not work in the present setting, and hence the necessity of adopting a global perspective and of using the results presented in this work.

\section{Appendix}

In this section, we briefly recall some results we have used here and there in the proofs above.

First, we recall the concept of Fourier transform (or characteristic functional) of a measure: refer to e.g. \cite[Definition 7.13.2]{bogachev2007measure}.
\begin{definition}\label{def: Fourier}
    Let $\X$ be locally convex space and let $\mu$ be a probability measure defined on $\Sigma(\X')$, namely, the $\sigma$-algebra generated by all continuous and linear functionals on $\X$. The Fourier transform of $\mu$ is the function $\hat{\mu}:\X\to\C$ so defined
    \begin{equation}
        \hat{\mu}(\phi) = \int_\R e^{it}\mu\circ \phi^{-1}(dt),\quad \phi\in \X'.
    \end{equation}
\end{definition}
The Fourier transform allows to distinguish different measures, namely it holds \cite[Lemma 7.13.5]{bogachev2007measure}:
\begin{lemma}\label{lemma: Fourier}
    If $\mu_1$ and $\mu_2$ are probability measures on $\Sigma(\X')$ and $\hat{\mu_1}=\hat{\mu_2}$, then one has $\mu_1=\mu_2$ on $\Sigma(\X')$. If both $\mu_1$ and $\mu_2$ are Radon on $\mathcal{B}(\X)$, then the equality holds on $\mathcal{B}(\X)$ too.
\end{lemma}

The next lemma from \cite{benth2023neural} has been crucial for the proof of Theorem \ref{thm: main}: we recall it here.
\begin{lemma}\label{lemma: rotation of hyperplanes}
Let $\X$ be a real Fr\'echet space. Let $\psi\in \X'$ be not identically zero. Then, for arbitrary $\gamma\in \X'$, the equation
\begin{equation*}
    \gamma = \psi \circ A
\end{equation*}
is solvable for some $A\in\mathcal{L}(\X)$.
\end{lemma}

\vskip 1cm
\noindent {\bf Conflict of Interest:} The authors declare that they have no conflict of interest.

\bibliography{literature}

@misc{HJ,
      title={Steps Toward Deep Kernel Methods from Infinite Neural Networks}, 
      author={Tamir Hazan and Tommi Jaakola},
      year={2015},
      eprint={1508.05133},
      archivePrefix={arXiv},
      primaryClass={cs.LG}
}

@misc{GUss,
      title={On Universal Approximation by Neural Networks with Uniform Guarantees on Approximation of Infinite Dimensional Maps}, 
      author={William H. Guss and Ruslan Salakhutdinov},
      year={2019},
      eprint={1910.01545},
      archivePrefix={arXiv},
      primaryClass={cs.LG}
}

@article{benth2023neural,
  title={Neural networks in Fr{\'e}chet spaces},
  author={Benth, Fred Espen and Detering, Nils and Galimberti, Luca},
  journal={Annals of Mathematics and Artificial Intelligence},
  volume={91},
  number={1},
  pages={75--103},
  year={2023},
  publisher={Springer}
}

@ARTICLE{Williams,
  author={Williams, C. K. I.},
  journal={Advances in neural information processing systems,}, 
  title={Computing with infinite networks}, 
  year={1997},
  pages={295-301}
  }

@ARTICLE{ChoSaul,
  author={Cho, Youngmin and Saul, Lawrence K.},
  journal={Advances in neural information processing systems,}, 
  title={Kernel methods for deep learning}, 
  year={2009},
  pages={342-350}
  }

@ARTICLE{MhaskarHahm,
  author={Mhaskar, H. N. and Hahm, Nahmwoo}, 
  journal={Neural Computation}, 
  title={Neural Networks for Functional Approximation and System Identification}, 
  year={1997},
  volume={9},
  number={1},
  pages={143-159},
  doi={10.1162/neco.1997.9.1.143}}

@book{Neal,
author = {Neal, Radford M.},
address = {New York},
title = {Bayesian {L}earning for {N}eural {N}etworks},
series = {Lecture {N}otes in {S}tatistics: 118},
publisher = {Springer Science+Business Media},
year = {1996}
}

@book{conway2010,
author = {Conway, John B.},
address = {New York},
edition = {2nd},
isbn = {9781441930927},
language = {eng},
publisher = {Springer Science+Business Media},
series = {Graduate {T}exts in {M}athematics; 96},
title = {A {C}ourse in {F}unctional {A}nalysis},
year = {2010}
}

@article{Cybenko1989,
	Author = {Cybenko, G. },
	Da = {1989/12/01},
	Doi = {10.1007/BF02551274},
	Id = {Cybenko1989},
	Isbn = {1435-568X},
	Journal = {Mathematics of Control, Signals and Systems},
	Number = {4},
	Pages = {303--314},
	Title = {Approximation by superpositions of a sigmoidal function},
	Ty = {JOUR},
	Url = {https://doi.org/10.1007/BF02551274},
	Volume = {2},
	Year = {1989}}

@article{galimberti2024neural,
  title={Neural networks in non-metric spaces},
  author={Galimberti, Luca},
  journal={arXiv preprint arXiv:2406.09310},
  year={2024}
}

@article{hornik1991approximation,
  title={Approximation capabilities of multilayer feedforward networks},
  author={Hornik, Kurt},
  journal={Neural networks},
  volume={4},
  number={2},
  pages={251--257},
  year={1991},
  publisher={Elsevier}
}

@article{lim2025score,
  title={Score-based diffusion models in function space},
  author={Lim, Jae Hyun and Kovachki, Nikola B and Baptista, Ricardo and Beckham, Christopher and Azizzadenesheli, Kamyar and Kossaifi, Jean and Voleti, Vikram and Song, Jiaming and Kreis, Karsten and Kautz, Jan and others},
  journal={Journal of Machine Learning Research},
  volume={26},
  number={158},
  pages={1--62},
  year={2025}
}

@article{pidstrigach2024infinite,
  title={Infinite-dimensional diffusion models},
  author={Pidstrigach, Jakiw and Marzouk, Youssef and Reich, Sebastian and Wang, Sven},
  journal={Journal of Machine Learning Research},
  volume={25},
  number={414},
  pages={1--52},
  year={2024}
}

@book{schaefer1971topological,
  title={Topological Vector Spaces},
  author={Schaefer, H.H.},
  isbn={9783540053804},
  lccn={lc75156262},
  series={Elements of mathematics / N. Bourbaki},
  url={https://books.google.com/books?id=AnWCAAAAIAAJ},
  year={1971},
  publisher={Springer}
}

@book{bogachev2007measure,
  title={Measure Theory},
  author={Bogachev, V.I.},
  number={v. 2},
  lccn={2006933997},
  series={Measure Theory},
  url={https://books.google.com/books?id=GbDazAEACAAJ},
  year={2007},
  publisher={Springer}
}

@article{li2020fourier,
  title={Fourier neural operator for parametric partial differential equations},
  author={Li, Zongyi and Kovachki, Nikola and Azizzadenesheli, Kamyar and Liu, Burigede and Bhattacharya, Kaushik and Stuart, Andrew and Anandkumar, Anima},
  journal={arXiv preprint arXiv:2010.08895},
  year={2020}
}

@article{bilokopytov2026universal,
  title={A universal approximation theorem and its applications to vector lattice theory},
  author={Bilokopytov, Eugene and Xanthos, Foivos},
  journal={Journal of Mathematical Analysis and Applications},
  pages={130632},
  year={2026},
  publisher={Elsevier}
}

@article{ismailov2026shallow,
  title={On shallow feedforward neural networks with inputs from a topological space: VE Ismailov},
  author={Ismailov, Vugar E},
  journal={Annals of Mathematics and Artificial Intelligence},
  pages={1--12},
  year={2026},
  publisher={Springer}
}

@misc{CuchieroSchmockerTeichmann,
  url = {https://arxiv.org/pdf/2306.03303.pdf},
  title={Global universal approximation of functional input maps on weighted spaces},
  author={Christa Cuchiero and Philipp Schmocker and Josef Teichmann},
  publisher={arXiv},
  year={2023}
}

@article{Hypertransformer,
  title={Designing universal causal deep learning models: The geometric (Hyper)transformer},
  author={Beatric Acciaio and Anastasis Kratsios and Gudmund Pammer},
  journal={Mathematical Finance},
  year={2024},  
  volume = {34},
  pages={671-735}
}

@article{wang2021learning,
  title={Learning the solution operator of parametric partial differential equations with physics-informed DeepONets},
  author={Wang, Sifan and Wang, Hanwen and Perdikaris, Paris},
  journal={Science advances},
  volume={7},
  number={40},
  pages={eabi8605},
  year={2021},
  publisher={American Association for the Advancement of Science}
}

@article{FUNAHASHI1989183,
title = {On the approximate realization of continuous mappings by neural networks},
journal = {Neural Networks},
volume = {2},
number = {3},
pages = {183-192},
year = {1989},
issn = {0893-6080},
doi = {https://doi.org/10.1016/0893-6080(89)90003-8},
url = {https://www.sciencedirect.com/science/article/pii/0893608089900038},
author = {Ken-Ichi Funahashi}
}

@article{HORNIK1989359,
title = {Multilayer feedforward networks are universal approximators},
journal = {Neural Networks},
volume = {2},
number = {5},
pages = {359-366},
year = {1989},
issn = {0893-6080},
doi = {https://doi.org/10.1016/0893-6080(89)90020-8},
url = {https://www.sciencedirect.com/science/article/pii/0893608089900208},
author = {Kurt Hornik and Maxwell Stinchcombe and Halbert White}
}

@article{LESHNO1993861,
title = {Multilayer feedforward networks with a nonpolynomial activation function can approximate any function},
journal = {Neural Networks},
volume = {6},
number = {6},
pages = {861-867},
year = {1993},
issn = {0893-6080},
doi = {https://doi.org/10.1016/S0893-6080(05)80131-5},
url = {https://www.sciencedirect.com/science/article/pii/S0893608005801315},
author = {Moshe Leshno and Vladimir Ya. Lin and Allan Pinkus and Shimon Schocken}
}

@ARTICLE{392253,
  author={Tianping Chen and Hong Chen},
  journal={IEEE Transactions on Neural Networks}, 
  title={Universal approximation to nonlinear operators by neural networks with arbitrary activation functions and its application to dynamical systems}, 
  year={1995},
  volume={6},
  number={4},
  pages={911-917},
  doi={10.1109/72.392253}}

@ARTICLE{286886,
  author={Chen, T. and Chen, H.},
  journal={IEEE Transactions on Neural Networks}, 
  title={Approximations of continuous functionals by neural networks with application to dynamic systems}, 
  year={1993},
  volume={4},
  number={6},
  pages={910-918},
  doi={10.1109/72.286886}}

@ARTICLE{76498,
  author={Sandberg, I.W.},
  journal={IEEE Transactions on Circuits and Systems}, 
  title={Approximation theorems for discrete-time systems}, 
  year={1991},
  volume={38},
  number={5},
  pages={564-566},
  doi={10.1109/31.76498}}

@article{10.1093/imatrm/tnac001,
    author = {Lanthaler, Samuel and Mishra, Siddhartha and Karniadakis, George E},
    title = "{Error estimates for DeepONets: a deep learning framework in infinite dimensions}",
    journal = {Transactions of Mathematics and Its Applications},
    volume = {6},
    number = {1},
    year = {2022},
    month = {03},
    issn = {2398-4945},
    doi = {10.1093/imatrm/tnac001},
    url = {https://doi.org/10.1093/imatrm/tnac001},
    note = {tnac001},
    eprint = {https://academic.oup.com/imatrm/article-pdf/6/1/tnac001/42785544/tnac001.pdf},
}

@article{LiDeepONets,
	author = {Lu, Lu and Jin, Pengzhan and Pang, Guofei and Zhang, Zhongqiang and Karniadakis, George Em},
	date = {2021/03/01},
	doi = {10.1038/s42256-021-00302-5},
	id = {Lu2021},
	isbn = {2522-5839},
	journal = {Nature Machine Intelligence},
	number = {3},
	pages = {218--229},
	title = {Learning nonlinear operators via DeepONet based on the universal approximation theorem of operators},
	url = {https://doi.org/10.1038/s42256-021-00302-5},
	volume = {3},
	year = {2021}}
\bibliographystyle{abbrv}

\end{document}